\documentclass[12pt,reqno]{amsart}

\usepackage{amssymb, color, xypic}

\newtheorem{theorem}{Theorem}[section]
\newtheorem{proposition}[theorem]{Proposition}
\newtheorem{lemma}[theorem]{Lemma}

\numberwithin{equation}{section}

\renewcommand{\H}{\mathbb{H}}

\newcommand{\C}{\mathbb{C}}
\newcommand{\SA}{\mathcal{A}}

\newcommand{\SC}{\mathcal{C}}
\newcommand{\SD}{\mathcal{D}}
\newcommand{\SE}{\mathcal{E}}

\newcommand{\SL}{\mathcal{L}}
\newcommand{\SM}{\mathcal{M}}

\newcommand{\SO}{\mathcal{O}}
\newcommand{\SP}{\mathcal{P}}
\newcommand{\ST}{\mathcal{T}}
\newcommand{\cS}{\mathcal{S}}

\newcommand{\PP}{\SP\SP}
\newcommand{\eps}{\epsilon}
\begin{document}

\baselineskip=15pt

\title[Isomonodromic deformations of irregular connections and
stability]{Isomonodromic deformations of irregular connections and
stability of bundles}

\author[I. Biswas]{Indranil Biswas}

\address{School of Mathematics, Tata Institute of Fundamental
Research, Homi Bhabha Road, Bombay 400005, India}

\email{indranil@math.tifr.res.in}

\author[V. Heu]{Viktoria Heu}

\address{Institut de Recherche Math\'ematique Avanc\'ee, 7 rue 
Ren\'e-Descartes, 67084 Strasbourg Cedex, France}

\email{heu@math.unistra.fr}

\author[J. Hurtubise]{Jacques Hurtubise}

\address{Department of Mathematics, McGill University, Burnside
Hall, 805 Sherbrooke St. W., Montreal, Que. H3A 0B9, Canada}

\email{jacques.hurtubise@mcgill.ca}



\date{}

\maketitle

\begin{abstract} Let $G$ be a reductive affine algebraic group defined over
$\mathbb C$, and let $\nabla_0$ be  a meromorphic $G$-connection  on a holomorphic $G$-bundle $E_0$, over a smooth complex curve $X_0$, with polar locus $P_0 \subset X_0$.  We assume that $\nabla_0$ is irreducible in the sense that it does not factor through some proper parabolic subgroup of $G$. We consider the universal isomonodromic deformation $(E_t\to X_t, \nabla_t, P_t)_{t\in \mathcal{T}}$ of $(E_0\to X_0, \nabla_0, P_0)$, where $\mathcal{T}$ is a certain quotient of a certain framed Teichm\"uller space we describe. We show that if the genus $g$ of $X_0$ satisfies $g\geq 2$, then for a general parameter $t\in \mathcal{T}$, the $G$-bundle $E_t\to X_t$ is stable.  For $g\geq 1$, we are able to show that for a general parameter $t\in \mathcal{T}$, the $G$-bundle $E_t\to X_t$ is semistable. 
\end{abstract}

\section{Introduction}

The natural correspondence between a flat connection on a principal bundle defined over a variety 
and its monodromy representation is a recurrent theme in mathematics, with a long 
history, as evidenced by the name, Riemann--Hilbert problem, given to one of the core 
questions of the subject. This basic problem consists in asking when a representation 
of the fundamental group of a punctured Riemann sphere can be realized by a flat 
connection on a holomorphically trivial bundle, with simple poles at the punctures; the 
answer, which is most of the time, but not always (\cite{Plemelj}, \cite{Dekkers}, 
\cite{Bolibruch1}, \cite{Bolibruch2}, \cite{Kostov}), is in itself an interesting 
chapter of the history of mathematics.

If one relaxes the condition of triviality, and asks whether the representation can be realized on a principal bundle, then the answer is always yes, and indeed the correspondence is quite natural. The question then becomes that of whether the bundle can be made trivial, either by some twists at the punctures (Schlesinger transformations) or by deforming the location of the punctures (isomonodromic deformations). The deformation theoretic version of the Riemann--Hilbert problem becomes:

\emph{Given a logarithmic connection 
$(E_0\, ,\nabla_0)$
on $\mathbb{P}^1_{\mathbb C}$ with polar divisor $D_0$ of degree
$n$, is there a point $t$ of the Teichm\"uller space $\mathrm{Teich}_{0,n}$ such that
the underlying holomorphic vector bundle $E_t\,=\,\mathcal{E}\vert_{\mathbb{P}^1_{\mathbb C}
\times\{t\}}$ in the universal isomonodromic deformation
$(\mathcal{E}\, , \nabla)$ for $(E_0\, ,\nabla_0)$ is trivial?}

A partial answer to this question is given, in the case of vector bundles of rank two, 
by the following theorem of Bolibruch:

\begin{theorem}[{\cite{Bolibruch3}}]\label{thmBolibruch}
Let $(E_0\, ,\nabla_0)$ be an irreducible trace--free logarithmic rank two connection
with $n\,\geq\, 4$ poles on $\mathbb{P}^1_{\mathbb C}$ such that each singularity is
resonant. There is a proper closed complex analytic subset $\mathcal{Y}\,\subset\,
\mathrm{Teich}_{0,n}$ such that for all $t\,\in\, \mathrm{Teich}_{0,n}\setminus
\mathcal{Y}$, the holomorphic vector bundle $E_t\,=\,\mathcal{E}\vert_{\mathbb{P}^1_{
\mathbb C}\times\{t\}}$ underlying the universal isomonodromic deformation
$(\mathcal{E}\, , \nabla)$ of $(E_0,\nabla_0)$ is trivial. 
\end{theorem}

In \cite{Viktoria1}, it is shown that the resonance condition in Theorem 
\ref{thmBolibruch} is redundant.

This gives an indication for the Riemann sphere; one can actually consider a similar 
problem for an arbitrary Riemann surface. Indeed, triviality of a vector bundle over the 
Riemann sphere is equivalent to being semi-stable of degree zero. On a general Riemann 
surface, the question of whether one can realize a representation by a semi-stable vector bundle of
degree zero was considered in \cite{Helene1,Helene2}. The deformation 
version, whether a logarithmic connection on a bundle over an arbitrary Riemann surface 
admits an isomonodromic deformation to a logarithmic connection on a stable or semi-stable bundle, was 
treated in \cite{BHH}; see also \cite{Viktoria1} for rank two. We recall from \cite{BHH},
\cite{Viktoria1}:

\begin{theorem}[{\cite{BHH}, \cite{Viktoria1}}]\label{R-O}
Let $X$ be a compact connected Riemann surface of genus $g$, and let $D_0\,
\subset\, X $ be an ordered subset of it of cardinality $n$. Let $G$ be a
reductive affine algebraic group defined over
$\mathbb C$. Let $E_G$ be a holomorphic principal $G$--bundle on $X$ and
$\nabla$ a logarithmic connection on $E_G$ with polar divisor $D_0$. Let
$(\mathcal{E}_G\, , \nabla)$ be the universal isomonodromic deformation
of $(E_G\, ,\nabla_0)$ over the universal Teichm\"uller curve
$\tau\,:\, (\mathcal{X}\, , \mathcal{D})\,\longrightarrow\, \mathrm{Teich}_{g,n}$.
For any point $t\, \in\, \mathrm{Teich}_{g,n}$, the restriction
$\mathcal{E}_G\vert_{\tau^{-1}(t)}\,\longrightarrow\, {\mathcal X}_t\,:=\,
\tau^{-1}(t)$ will be denoted by $\mathcal{E}^t_G$.

\begin{enumerate}
\item Assume that $g\, \geq\, 2$ and $n\,=\, 0$. Then there is a
closed complex analytic subset $\mathcal{Y}\, \subset\,
\mathrm{Teich}_{g,n}$ of codimension at least $g$ such that for any $t\,\in\,
\mathrm{Teich}_{g,n} \setminus \mathcal{Y}$, the holomorphic principal $G$--bundle
$\mathcal{E}^t_G\,\longrightarrow\, {\mathcal X}_t$ is semistable.

\item Assume that $g\,\geq\, 1$, and if $g\,=\, 1$, then $n\, >\, 0$. Also, assume
that the initial monodromy representation for $\nabla $ at $t=0$ does not factor through some proper
parabolic subgroup of $G$. Then there is a
closed complex analytic subset $\mathcal{Y}'\, \subset\,
\mathrm{Teich}_{g,n}$ of codimension at least $g$ such that for any $t\,\in\,
\mathrm{Teich}_{g,n} \setminus \mathcal{Y}'$, the holomorphic principal $G$--bundle
$\mathcal{E}^t_G$ is semistable.

\item Assume that $g\,\geq\, 2$. Assume that the monodromy representation for
$\nabla $ at $t=0$ does not factor through some proper parabolic subgroup of $G$. Then there
is a closed complex analytic subset $\mathcal{Y}''\, \subset\,\mathrm{Teich}_{g,n}$ of
codimension at least $g-1$ such that for any $t\,\in\, \mathrm{Teich}_{g,n}\setminus
\mathcal{Y}''$, the holomorphic principal $G$--bundle $\mathcal{E}^t_G$ is stable.
\end{enumerate}
\end{theorem}

Our aim here is to extend this result to connections with irregular singularities, that 
is connections with higher order poles. Let us consider a triple $$(E_G \longrightarrow 
X\, , D\, ,\nabla), $$ where $E_G$ is a holomorphic principal $G$--bundle over a 
compact connected Riemann surface $X$, and $\nabla$ is an integrable holomorphic 
connection on $E_G$, with possibly irregular singularities bounded by a divisor $D$ on 
$X$; that is, if $\nabla$ has poles of order $n_i$ at points $p_i$ of $X$, we set $D\,=\, 
\sum_{i=1}^m n_ip_i$, and let $D_0\,= \,\sum_{i=1}^m p_i$ denote the reduced divisor. We 
will suppose that the leading order term (i.e., coefficient of $z^{-n_i}$) of the connection
at $p_i$ 
is conjugate to a regular semisimple element $h_{i,-n_i}$ of a fixed Cartan subalgebra 
$\mathfrak h$ of the Lie algebra $\mathfrak g$ of $G$. By a gauge transformation at the 
poles, the polar part of the connection can be conjugated to
$$ h_i(z) dz\,=\, (h_{i,-n_i}z^{-n_i} + h_{i,-n_i+1}z^{-n_i+1}+\cdots +h_{i,-1}z^{-1}) dz\, .
$$ 
If one allows a formal gauge transformation, then the connection itself can be put in this form at the puncture; the power series that effects this transformation though does not typically converge. Instead, there is additional monodromy data, given by Stokes matrices \cite{JMU}. A good introduction to the theory can be found in \cite{Sa}, and the more advanced results we need have been established in \cite{Boa1,Boa2}. We now give a brief outline of the basic ideas. 

For each irregular singular point, one chooses disks $\Delta_i$ centered at $p_i$, $1\,\leq\,
i\,\leq\, m$. On $\Delta_i$, as noted, one has a formal solution 
$$H_i(z) \,=\, \exp(\int h_i (z) dz)$$
with a monodromy $\mu_i \,=\, \exp( 2\pi\sqrt{-1} h_{i,-1} )$; one partitions
the disk into $2n_i-2$ angular sectors $\cS_{i,j}$ determined by the $ h_{i,-n_i}$. 
Associated to the intersections $\cS_{i,2j}\cap \cS_{i,2j+1}$, there is a fixed (independent of $j$) unipotent radical $U_{+,i}$ of a Borel subgroup associated to $\mathfrak h$; to the intersections $\cS_{i,2j+1}\cap \cS_{i,2j+2}$ one has associated the opposite unipotent radical $U_{-,i}$. We choose a base point $q_i$ in $\cS_{i,1}$.

One can then consider on each sector actual integrals $g_{i,j}(z) \,\in\, G$ of
the connection asymptotic to the $H_i$, and passing from the sector $\cS_{i,2j}$ to
$\cS_{i,2j+1}$, the two solutions are related by Stokes factors $u_{+,i,j}$ lying in
$U_{+,i}$. In passing from $\cS_{i,2j+1}$ to $\cS_{i,2j+2}$, the two solutions are related by Stokes factors $u_{-,i,j}$ lying in $U_{-,i}$. The monodromy of the connection around the singularity is the product 
$$\rho_i \,= \,\mu_i u_{-,i, n-1}\cdots u_{+,i,2} u_{-,i,1}u_{+,i,1}\, .$$
This monodromy and its decomposition into Stokes factors is defined up to the action of a torus.

For the deformations, one has a generalized Teichm\"uller space 
$\mathrm{Teich}_{\mathfrak h, g,m}$, which combines the standard $\mathrm{Teich}_{g,m}$ 
with the irregular polar parts. Note that this combines parameters on the curves with parameters associated to the group; in 
addition to the standard Teichm\"uller parameters of the punctured curve, one considers 
the extra parameters of the ``irregular type'', realized as the formal singularity 
$H_i(z) \,=\, \exp(\int h_i (z) dz)$. More details can be found below.

Lying above this deformation on the base, there is a theory of isomonodromic 
deformations of such connections, generalizing the one we have for the logarithmic 
case. Over the base parameters, in particular $H_i(z) \,=\,\exp(\int h_i (z) dz)$, 
which becomes an Abelian transition function at the puncture, one fixes the Stokes 
factors $u_{\pm,i, j}$ at the irregular singularities, and the representation 
$\pi_1(X\setminus D_0)\,\longrightarrow \,\text{GL}(n,\C)$ of the fundamental group. Fixing such
isomonodromy 
data gives a lift of the Teichm\"uller deformations to a deformation of singular 
connections. By Malgrange's theorem, such isomonodromic deformations exist, and determine the connection up to gauge transformations \cite{Ma} \cite{Viktoria2}. Our aim will be to show:

\begin{theorem}\label{Result} 
Assume
that the monodromy representation for $\nabla_0$ is irreducible in the sense that it does not factor through some proper
parabolic subgroup of $G$. 
\begin{enumerate}
\item If $g\,\geq\, 1$, then there is a
closed complex analytic subset $\mathcal{Y}\, \subset\,
\mathrm{Teich}_{\mathfrak h, g,m}$ of codimension at least $g$ such that for any $t\,\in\,
\mathrm{Teich}_{\mathfrak h, g,m} \setminus \mathcal{Y}$, the holomorphic principal $G$--bundle
$\mathcal{E}^t_G$ is semistable.

\item If $g\,\geq\, 2$, then there
is a closed complex analytic subset $\mathcal{Y}'\, \subset\,\mathrm{Teich}_{\mathfrak h, g,m}$ of
codimension at least $g-1$ such that for any $t\,\in\, \mathrm{Teich}_{\mathfrak h, g,m}\setminus
\mathcal{Y}'$, the holomorphic principal $G$--bundle $\mathcal{E}^t_G$ is stable.
\end{enumerate}
\end{theorem}

\section{The base space}

We will describe the space $\mathrm{Teich}_{\mathfrak h, g,m}$.

The Teichm\"uller space $\mathrm{Teich}_{ g,m}$ for genus $g$ curves with $m$ marked 
points is a contractible complex manifold of complex dimension $3g-3+m$, 
assuming that $3g-3+m\, >\, 0$. We first build a framed 
Teichm\"uller space. If the singularity divisor is $D\,=\, \sum_{i=1}^m n_ip_i$, we 
can enrich the Teichm\"uller space $\mathrm{Teich}_{ g,m}$ by adding to each point 
$(\Sigma,\, \sum_{i=1}^m p_i)$ of $\mathrm{Teich}_{g,m}$, the additional data of a 
coordinate $z_i$ centered at $p_i$, defined to order $n_i-1$ inclusively, for all
$n_i\,>\,1$. We note that this additional data at $p_i$ is the choice of an
isomorphism of the algebra $m_{p_i}/m^{n_i}_{p_i}$ with $z{\mathbb C}[z]/
z^{n_i}{\mathbb C}[z]$, where $m_{p_i}$ is the ring of holomorphic functions defined
around $p_i$ that vanish at $p_i$. The Teichm\"uller space $\mathrm{Teich}_{ g,m}$
together with the above data produce a framed Teichm\"uller space 
$\mathrm{FTeich}_{g,m,n_1,\cdots ,n_m}$.

Now consider the extra data of the parameters $h_{i,j}$ of the polar parts of the connection. 
We set our framed Teichm\"uller space for deformations of the irregular part of
the connections plus punctured curves to simply be a product:
$$\mathrm{FTeich}_{{\mathfrak h}, g,m,n_1,\cdots ,n_m} \,=\, \mathrm{FTeich}_{g,m,n_1,
\cdots ,n_m}\times \prod_{i=1}^m ({\mathfrak h_0} \times {\mathfrak h }^{n_i-1})\, .$$
 
Our desired space of deformations $\mathrm{Teich}_{\mathfrak h, g,m}$ will be the 
quotient of this space by the groups of germs of diffeomorphisms of neighborhoods 
of $p_i$ which fix $p_i$, acting diagonally on the factors. As the action at each 
puncture is on truncated power series, one need only act by groups of jets $$J_{p_i, 
n_i} = \left\{ z\mapsto a_1z+ a_2z^2+\ldots +a_{n_i}z^{n_i}~\middle|~a_j\in \mathbb{C}, 
~ a_1\neq 0\right\};$$ one has $$ \mathrm{Teich}_{\mathfrak h, g,m} = 
\mathrm{FTeich}_{{\mathfrak h}, g,m,n_1,\cdots ,n_m}/_{\prod_iJ_{p_i,n_i}} .$$ (In fact, 
one would want to go to a universal cover, but for our purposes, this is not 
necessary as we are just considering the local deformations.)

Let us now see what this gives us for infinitesimal deformations. The tangent 
space of $\mathrm{T}_X(-D_0)$ at any element $(X\, , D_0)\, \in\, 
\mathrm{T}_X(-D_0)$ is $$\mathrm{H}^1(X,\, \mathrm{T}_X(-D_0))\, .$$ We note that a 
$1$- cocycle $v$ can be thought of as giving an infinitesimal deformation of the 
coordinate changes from one patch to another; the $1$-coboundaries must be taken 
with values in the vector fields vanishing at $D_0$. Such a coboundary, however, 
affects the form of the irregular polar parts at $p_i$.

Indeed, these are not well defined in themselves, as they are acted on by 
diffeomorphisms of the curve fixing $p_i$. This must be taken into account in 
the deformation theory. Consider an infinitesimal local diffeomorphism of the 
curve given at the puncture $z\,=\,0$ by a vector field $v(z) \partial/\partial 
z$. As we are considering punctured curves, we want $v(0) \,=\, 0$. The changes 
in the function $H_i(z)$ caused by a change in parametrization, infinitesimally 
a vector field, should be considered as trivial: in other words,
$$
H_i(z+\eps v(z)) \,=\, H_i(z) (1+ \eps H_i^{-1}(z)H_i'(z) v(z)) \,=\,
H_i(z) (1+ \eps h_i(z)v(z))\, .
$$

Thus, for our deformations, we will be interested in the complex 
$$\SC\,:\, \mathrm{T}_X(-D_0)\, \buildrel{F}\over{ \longrightarrow}\,\PP\, = \,
\SO_{D-D_0} \otimes \mathfrak h \,=\,\bigoplus_i\mathfrak h^{\oplus n_i-1}\, ;$$
the second sheaf is a sum of skyscraper sheaves supported at the points of $D_0$; the
homomorphism $F$ sends a vector field $v$ around $p_i$ to the
irregular polar part ($\PP$) of the contraction of $v$
with the connection matrix $h_i$: $$F\,:\, v(z) \,\longmapsto\, \PP((h_i(z) v(z))) .$$
The first order deformations of the marked curve are given by the cohomology group \linebreak $\mathrm{H}^1(X,\, \mathrm{T}_X(-D_0) )$; adding in the deformations of the irregular polar parts gives us the global hypercohomology group 
$$\H^1(X,\, \SC) \, = \, \mathrm{T}_{(X,D,H)}\mathrm{Teich}_{\mathfrak h, g,m} \, .$$

We note that $\H^1(X,\, \SC)$ coincides with the space of admissible deformations in \cite{Boa3} 
of the irregular curve defined by the triple $(X,\, D_0,\, \bigoplus_i\mathfrak h^{\oplus 
n_i-1})$. In \cite{Boa3}, the space of objects, consisting of a Riemann surface, some marked 
points on it and irregular types at the marked points, are defined more intrinsically.

We have an exact sequence
$$\bigoplus_i \mathfrak h^{\oplus n_i-1} \, \longrightarrow \, \H^1(X,\,\SC) \, \longrightarrow \,
\mathrm{H}^1(X,\, \mathrm{T}_X(-D_0) ) \, .$$
The elements $\beta$ of $\mathrm{H}^1(X,\, \mathrm{T}_X(-D_0))$ encode extensions 
$$ 0 \, \longrightarrow \, \mathrm{T}_X(-D_0) \, \longrightarrow \, \ST \,\longrightarrow \, \SO_X \, \longrightarrow \, 0 \, .$$
This can be viewed as the tangent bundle to the infinitesimal one-parameter family of bundles represented by the element $\beta$, with the structure sheaf $\SO_X$ representing a trivial normal bundle. An element $\widehat\beta$ of $\H^1(X,\,\SC)$ mapping to $\beta$ encodes a bit more, namely a diagram

$$\begin{xy}\xymatrix{\, \,0\, \, \ar[r]&\, \,\mathrm{T}_X(-D_0)\, \, \ar[r]\ar[d]_{F}&\, \, \ST\, \,\ar[r]\ar[d]& \, \,\SO_X\, \,\ar[r]& 0 \\ &\, \,\PP\, \,\ar@{=}[r]&\, \, \PP\, .}\end{xy}$$

\section{Deforming the bundle}

The Lie algebra of $G$ will be denoted by $\mathfrak g$. Let 
$\text{ad}(E_G)\,=\, E_G\times^G{\mathfrak g}$ be the adjoint bundle for $E_G$ 
over $X$. Let $\text{At}(E_G)$ denote the Atiyah bundle for $E_G$; it fits in 
the Atiyah exact sequence over $X$
$$
0\, \longrightarrow\, \text{ad}(E_G) \, \longrightarrow\,\text{At}(E_G)
\, \longrightarrow\, \mathrm{T}_X \, \longrightarrow\, 0
$$
\cite{At}. The Atiyah bundle for $E_G$ represents over the base the 
$G$--invariant vector fields on the principal $G$--bundle $E_G$; the subbundle 
of invariant vector fields tangent to the fibers is ${\rm ad}(E_G)$. The Atiyah 
exact sequence produces a short exact sequence
$$
0\, \longrightarrow\, \text{ad}(E_G) \, \longrightarrow\,\text{At}_{D_0}\,:=\,
\text{At}_{D_0}(E_G)\, \longrightarrow\,\mathrm{T}_X(-D_0) \, \longrightarrow\, 0\, ,
$$
where $D_0$ is the reduced singular locus of the connection.
In \cite{BHH} it is shown that the deformations of the logarithmic connection, over a
curve $X$ that is also being deformed, are parametrized by
$\mathrm{H}^1(X,\, \text{At}_{D_0})$.

To deal with the higher 
order poles, we need to consider the sheaf $\mathrm{At}_{D_0}( D-D_0)$ of 
meromorphic sections of $\mathrm{At}_{D_0}$ with poles living only in the ${\rm 
ad}(E_G)$ factor, bounded by $D-D_0$:
$$0\, \longrightarrow \, \mathrm{ad}(E_G)(D-D_0)\, \longrightarrow \, \mathrm{At}_{D_0}( D-D_0) \, \longrightarrow \, \mathrm{T}_X(-D_0) \, \longrightarrow \, 0\, .$$

Now let us consider deformations of these. We cover our Riemann surface away from the 
punctures with Stokes sectors $S_{i,j}$, as well as other contractible open sets 
$V_\nu$; choose flat trivializations on these sets, with the ones on Stokes sectors 
being compatible with the formal asymptotics. The transition functions on the bundle 
for these trivializations are then constants, with those between the Stokes sectors 
being the Stokes matrices. For the puncture, we have the transition functions 
$H_i(z)$. Re-label the Stokes sectors as being in the set of $V_\nu$; we then have 
constant transition functions $\Theta_{\nu_1,\nu_2}$ away from the puncture, and 
$H_i(z)$ at the puncture. Now take a variation $H_i(z)(1 + \eps \int k_i(z))$ and a 
cocycle $v_{\nu_1,\nu_2}$ for $\mathrm{T}_X(-D_0)$ , which corresponds to 
infinitesimal displacements of the coordinate patches with respect to each other; 
these together arise from a class $\hat \beta$ in $\H^1(X,\,\SC)$.

We are, in our isomonodromic deformations, deforming the bundle above the curve by 
keeping the same $\Theta_{\nu_1,\nu_2}$, and modifying the transition function at the 
puncture by $$H_i(z)(1 + \eps \int k_i(z))\, .$$ As a deformation of the Atiyah bundle, the 
former consists of considering the mapping $$\nabla\,: \,\mathrm{T}_X(-D_0)\, 
\longrightarrow \, \mathrm{At}_{D_0}(D-D_0)\, ,$$ and taking the induced map on the 
cocycles, i.e. taking $\nabla(v_{\nu_1,\nu_2})$ as a cocycle for 
$$\mathrm{At}_{D_0}(D-D_0)\, ,$$ which, as we are away from the punctures, we can take to 
be a cocycle in for $\mathrm{At}_{D_0}$. To this, we add the element $ k_i(z)$ as a 
cocycle for the deformation of the transition function from the disk around the 
puncture to the Stokes sectors, for the subbundle $\mathrm{ad}(E_G)$ of 
$\mathrm{At}_{D_0}$; the sum of the cocycles gives a class $\gamma$ in 
$\mathrm{H}^1(X,\, \mathrm{At}_{D_0})$.
 
As for the deformation of the curves, an element $\gamma$ of $\mathrm{H}^1(X,\, \mathrm{At}_{D_0} )$ defines an extension
 $$ 0\,\longrightarrow \, \mathrm{At}_{D_0}\, \longrightarrow \,\SA \, \longrightarrow\,\SO_X\,\longrightarrow\, 0\,,$$
 mapping to corresponding extensions of $\mathrm{T}_X(-D_0)$, and so gives a diagram 
$$
 \begin{xy}\xymatrix{
 &\, \,\mathrm{ad}(E_G)\, \,\ar@{=}[r]\ar[d] &\, \,\mathrm{ad}(E_G)\, \,\ar[d] \\ 0\ar[r]& \, \,\mathrm{At}_{D_0}\, \,\ar[d] \ar[r]& \, \,\SA \, \,\ar[r]\ar[d] &\, \, \SO_X\, \,\ar[r]\ar@{=}[d]& \, \, 0 \, \, \\ \, \,0\, \,\ar[r]&\, \,\mathrm{T}_X ( -D_0)\, \,\ar[r]& \, \,\ST\, \, \ar[r] & \, \,\SO_X\, \,\ar[r]& \, \, 0 \, .}\end{xy}$$
 This represents over $\eps\,=\, 0$ the $G$--invariant vector fields on our first order extension, the $\SO_X$--quotient being the normal bundle.

\section{Deformations of reductions}

\subsection{Extending a reduction}

The stability of $G$-bundles concerns reductions to a parabolic subgroup: the bundle 
$E_G$ is stable (respectively, semistable) if for all its reductions $E_P$ to a parabolic subgroup $P$, 
the associated bundle
$$
\text{ad}(E_G)/\text{ad}(E_P)\,=\, E_P({\mathfrak g}/{\mathfrak p})
$$
has positive (respectively, non-negative) degree, where ${\mathfrak g}$ and ${\mathfrak p}$
are the Lie algebras of $G$ and $P$ respectively. If we want to ensure that the set of
non stable bundles is somehow small along the isomonodromic deformation, we must see how
reductions to a parabolic extend along a deformation, and in particular try to understand
the space of first order obstructions to such an extension.

Given a reduction $E_P$, we now have two Atiyah bundles $\mathrm{At}_{D_0}^G$ and $\mathrm{At}_{D_0}^P$
over the surface associated to $E_G$ and $E_P$ respectively. These fit into a diagram:
\begin{equation}\label{e6}
\begin{xy}\xymatrix{
& \, \,0\, \,\ar[d] & \, \,0\, \,\ar[d] \\
\, \,0\, \,\ar[r]&\, \, \text{ad}(E_P) \, \,\ar[r]\ar[d] & \, \,\mathrm{At}_{D_0}^P\, \, \ar[r]^{\beta}\ar[d]^{\xi}&\, \, \mathrm{T}X (-D_0)\, \, \ar[r]\ar@{=}[d] & \, \,0\, \,\\
\, \,0\, \,\ar[r]& \, \,\text{ad}(E_G)\, \, \ar[r] \ar[d]^{\omega_1}& \, \,\mathrm{At}_{D_0}^G\, \,\ar[r]^{\sigma}\ar[d]^{\omega} & \, \,\mathrm{T}X (-D_0) \ar[r]\, \, &\, \, \, 0\, .\\
& \, \,0\, \, & \quad \, \, 0\quad \,}
\end{xy}
\end{equation}

Now assume that the reduction to $P$ extends to first order along a first order deformation of the $G$-bundle over the curve. One then has extensions

\begin{equation} \label{deform-reductions}
\begin{xy}\xymatrix{
 \, \,0\, \,\ar[r] &\, \,\mathrm{At}_{D_0}^P \, \,\ar[r]\ar[d]& \, \,\SA^P \, \,\ar[r]\ar[d] & \, \,\SO_X\, \,\ar[r] \ar@{=}[d]&\, \,0\, \,\\
\, \,0\, \,\ar[r] &\, \,\mathrm{At}_{D_0}^G \, \,\ar[r]& \, \,\SA^G \, \,\ar[r] & \, \,\SO_X\, \,\ar[r]& \, \,0\, .
}
\end{xy}
\end{equation}
given by extension classes $\gamma^P\,\in\, \mathrm{H}^1(X ,\, \mathrm{At}_{D_0}^P)$
and $\gamma^G\,\in\, \mathrm{H}^1(X ,\, \mathrm{At}_{D_0}^G)$. One has the lemma

\begin{lemma}[{\cite[Lemma 3.1]{BHH}}]\label{obstruction}
The above extension classes $\gamma^P, \gamma^G$ are related by $$\gamma^G \,=\,
\xi(\gamma^P)\, .$$ Consequently, there is an obstruction to extending the reductions for
deformations $\gamma^G$ given by $\omega(\gamma^G)\,\in\, \mathrm{H}^1(X,\,
E_P({\mathfrak g}/{\mathfrak p}))$.
\end{lemma}

\subsection{A second fundamental form}

Assume now that there is a connection $\nabla$ on the bundle $E_G$. It does not of course, necessarily preserve the reduction to $E_P$. The failure to preserve $ E_P$ is measured by a second fundamental form: one takes the composition
\begin{equation}\label{sff} S(\nabla) \, = \, \omega\circ \nabla \, : \, \mathrm{T}_X(-D_0) \, \longrightarrow \, \mathrm{At}_{D_0} ( D-D_0) \, \longrightarrow \, E_P({\mathfrak g}/{\mathfrak p})(D-D_0) \, .\end{equation}
The connection preserves the reduction to $P$ if and only if $S(\nabla)\,=\,0$.

Assume that $E_P$ satisfies the condition that
$
S(\nabla)\,\not=\, 0\, .
$
We define some line bundles. Let 
$$
{\mathcal M}({D-D_0})\, \subset\, \mathrm{At}_{D_0}(D-D_0)
$$
be the holomorphic line subbundle generated by the image 
$ \nabla(\mathrm{T}_X(-D_0))$ in \eqref{sff}, and let
$$
{\mathcal L}({D-D_0})\, \subset\, E_P({\mathfrak g}/{\mathfrak p})(D-D_0)
$$
be the holomorphic line subbundle generated by the image 
$ \omega(\nabla(\mathrm{T}_X(-D_0)))$ in \eqref{sff}. More
precisely, ${\mathcal M}_{D-D_0}$ (respectively, ${\mathcal L}_{D-D_0}$) is the inverse
image in $\mathrm{At}_{D_0}(D-D_0)$ (respectively, $E_P({\mathfrak g}/{\mathfrak p})(D-D_0)$)
of the torsion part of the quotient $\mathrm{At}_{D_0}(D-D_0)/\nabla(\mathrm{T}_X(-D_0)$ (respectively,
$E_P({\mathfrak g}/{\mathfrak p})/(\omega\circ\nabla)(\mathrm{T}X(-D_0)))$. Set 
\begin{equation}\label{l}
{\mathcal M}\,=\, {\mathcal M}_{D-D_0}\cap \mathrm{At}_{D_0}\, , \quad \quad {\mathcal L}
\,=\, {\mathcal L}_{D-D_0}\cap E_P({\mathfrak g}/{\mathfrak p})\, .
\end{equation}

We then have the diagram of homomorphisms of line bundles, with the columns being exact:
$$
 \begin{xy}\xymatrix{
\, \, \mathrm{T}_X(-D)\, \,\ar[r]\ar[d]& \, \,{\mathcal M}\, \,\ar[r]\ar[d] &\, \, {\mathcal L}\, \,\ar[d]\\
\, \,\mathrm{T}_X(-D_0)\, \,\ar[r]\ar[d]& \, \,{\mathcal M}(D-D_0)\, \,\ar[r]\ar[d] & \, \,{\mathcal L}(D-D_0)\, \,\ar[d]\\
\, \,Q_1\, \,\ar[d]&\, \,Q_2\, \,\ar[d]&\, \,Q_3\, \,\ar[d]\\ \, \,0\, \,&\, \,0\, \,&\,\, \,0\, .
}\end{xy}
$$

Note that $Q_1$, $Q_2$ and $Q_3$ are isomorphic torsion sheaves supported on $D-D_0$.

\begin{lemma}\label{surj1}
The horizontal homomorphisms in this diagram induce surjective maps on the level of
first cohomology.
\end{lemma}

\begin{proof}
The proof consists in noting that the cokernels of each of these homomorphisms are
torsion sheaves.
\end{proof}

If one considers the homomorphism $\mathrm{T}_X(-D )\,\longrightarrow\, {\mathcal M}
\,\subset\, \mathrm{At}_{D_0}$ given by the 
connection, one has that ${\mathcal M}$ lies in the ``Cartan component" of the bundle to 
order $n_i-1$ at $p_i$, as it is a multiple of $h_i(z)$. For the sheaf $\PP$ of polar 
parts of the connection, let us consider the subsheaf $\PP_\parallel$ whose sections are 
multiples of $h_i(z)$; likewise, in our deformation space $\H^1(X,\,\SC)$, let us consider 
the subspace $\H^1_\parallel(X,\,\SC)$ of classes where the principal part is parallel to 
(i.e. a multiple of) $h_i(z)$.

\begin{proposition}\label{surj2} We have a diagram
$$
\begin{xy}\xymatrix{
\, \,\PP_\parallel\, \,\ar[r]\ar[d]&\, \, \mathrm{H}^0(Q_2)\, \,\ar[d]\\
\, \,\H^1_\parallel(X,\, \SC)\, \,\ar[r]\ar[d]& \, \,\mathrm{H}^1(X,\, {\mathcal M})\, \,\ar[d]\\
\, \,\mathrm{H}^1(X, \, \mathrm{T}_X(-D_0))\, \,\ar[r]&\, \, \,\mathrm{H}^1(X, \, {\mathcal M}(D-D_0))\, .
}\end{xy}
$$
The top horizontal homomorphism is an isomorphism, and the other two horizontal
homomorphisms are surjective.
\end{proposition}

\begin{proof}
On the top row, the components of $\PP_\parallel$ are exactly those of the torsion sheaf 
$Q$. On the bottom row, one has elements $\beta$ of $\mathrm{H}^1(X,\, \mathrm{T}_X(-D_0))$ mapped by 
$\nabla$ to $\mathrm{H}^1(X,\,{\mathcal M}(D-D_0))$. As argued in Lemma \ref{surj1}, this map is 
surjective.

One now wants to see that the top and bottom fit together correctly in the middle term. Let 
$\widehat\beta\,\in\, \H^1_\parallel(X,\,\SC)$ be represented by elements $k_i$ of 
$\PP_\parallel$ at each puncture, and a representative cocycle $\beta$ of $\mathrm{H}^1(X, \, \mathrm{T}_X(-D_0))$. If one turns $k_i$ in the natural way into a cocycle supported on a 
punctured disk $\Delta_i$ at $p_i$, it gives precisely the element of $\mathrm{H}^1(X,\,{\mathcal M})$ 
which is the coboundary of $k_i$ thought of as an element of $Q_2$. In turn, the cocycle 
$\beta$ is simply mapped to $\mathrm{H}^1(X,\,{\mathcal M})$ by the sheaf map; the total map from 
$\H^1_\parallel(X,\,\SC)$ is given by the sum of these two contributions, as in the 
definition of $\SA$ above. Since the top map is an isomorphism, and the bottom one is 
surjective, the middle map is also surjective.
\end{proof}

We now have a surjective map $\H^1_\parallel(X,\, \SC)\, \longrightarrow\,
\mathrm{H}^1(X,\, {\mathcal M})$, which, when mapped on to $\mathrm{H}^1(X,\, \mathrm{At}_{X_0})$,
defines the extension $\SA$. We saw in turn that the map $$\omega\,:\,
\mathrm{H}^1(X,\, \mathrm{At}_{X_0})\,\longrightarrow \,
\mathrm{H}^1(X,\, E_P({\mathfrak g}/{\mathfrak p}))$$ gave an obstruction to extending to first order a reduction to $P$. We have a diagram:
\begin{equation}
\begin{xy}\xymatrix{
\, \,\H^1_\parallel(X,\, \SC)\, \,\ar[r]& \, \,\mathrm{H}^1(X,\, {\mathcal M})\, \,\ar[r]^{\omega_{\mathcal L}}\ar[d]&\, \, \mathrm{H}^1(X,\, {\mathcal L})\, \,\ar[d]\\
&\, \,\mathrm{H}^1(X,\, \mathrm{At}_{X_0})\ar[r]^{\omega}\, \, &\, \,\, \mathrm{H}^1(X,\, E_P({\mathfrak g}/{\mathfrak p}))\, .
}\end{xy}
\end{equation}
 
We have, as in Proposition 4.3 of \cite{BHH}:

\begin{proposition}\label{prop1}
Let $\widehat\beta \in \H^1_\parallel(X,\, \SC)$ represent an isomonodromy deformation class, of a
connection with non-vanishing second fundamental form, yielding a class
$$\gamma_\SM\, \in \, \mathrm{H}^1(X,\, {\mathcal M})\, ,$$ and a class $$\gamma \,\in\,
\mathrm{H}^1(X,\, \mathrm{At}_{X_0})$$ representing the extension $\SA$. 
The obstruction $\omega (\gamma)$ to extending a reduction to $P$ factors through $\SL$, as $\omega_\SL(\gamma_\SM)\,\in\, \mathrm{H}^1(X,\, \SL)$, and if the bundle reduces to $P$ then $\omega_\SL(\gamma_\SM)\,\in\, \mathrm{H}^1(X,\, \SL)$ is also zero. 
\end{proposition}
The proof in essence works by taking the restriction of the Atiyah bundle and its extension which live above $\SL$.

We will want to estimate the dimension of the space spanned by the obstructions $\omega_L(\gamma_\SM)$, as this will give a bound on the codimensions of the stable locus, as explained in the next section.

\section{Harder-Narasimhan filtrations} 
 
Let as before $ \mathrm{Teich}_{\mathfrak h, g,m}$ be our Teichm\"uller space; over it we 
have, locally at least, a universal family $(\mathcal{C},\, \mathcal{D},\, \mathcal{H})$ 
whose fiber at $q$ is a curve $\mathcal{C}(q)$, a divisor $$\mathcal{D}(q)\,=\, 
\sum_in_ip_i(q)$$ and a collection of formal solutions $H_i(q)$. Over this in turn the 
isomonodromy process, described in section 3, gives a $G-$bundle $\SE_G\,\longrightarrow\,\SC$, equipped with a flat 
connection, with the appropriate polar behavior at $\SD$. For $\SE_G$, one has a 
Harder-Narasimhan filtration for families of $G$-bundles, as propounded in \cite{Nitsure} 
(see also \cite{Sh}); the filtration is trivial if and only if the bundle is semi-stable.

\begin{lemma}[\cite{Nitsure}] \label{lemNitsure}
Let $\mathcal{E}_G\,\longrightarrow\, \mathcal{C}\,\longrightarrow\,\mathcal{T}_{\mathfrak h,g,m}$
be as above. For each Harder--Narasimhan type $\kappa$, the set
$$
\mathcal{Y}_\kappa \,:=\, \{t\,\in\, \mathcal{T}_{\mathfrak h,g,m} ~\mid ~
\mathcal{E}_G\vert_{{\mathcal C}_t}\ \text{ is\ of\ type }\ \kappa\}
$$
is a (possibly empty) locally closed complex analytic subspace of $\mathcal{T}_{\mathfrak h,g,m}$.
More precisely, for each Harder--Narasimhan type $\kappa$, the union 
$\mathcal{Y}_{\leq \kappa}\,:=\,\bigcup_{\kappa'\leq \kappa}\mathcal{Y}_{\kappa'}$ is
a closed complex analytic subset of $\mathcal{T}_{\mathfrak h,g,m}$.
Moreover, the principal $G$--bundle
$$
\mathcal{E}_G\vert_{\tau^{-1}(\mathcal{Y}_\kappa)}\,\longrightarrow\,
\tau^{-1}(\mathcal{Y}_\kappa)
$$
possesses a canonical holomorphic 
reduction of structure group inducing the Harder--Narasimhan
reduction of $\mathcal{E}_G\vert_{{\mathcal C}_t}$ for every $t\,\in \,\mathcal{Y}_\kappa$.
\end{lemma}

It is our aim to show that all of these strata except that corresponding to the trivial filtration (and hence to semi-stable bundles) are of codimension $g$, by showing that there is a $g$-dimensional family of directions for which the reduction does not extend. 

\begin{theorem}\label{prop2} 
Assume
that the monodromy representation for $\nabla_0$ is irreducible in the sense that it does not factor through some proper
parabolic subgroup of $G$. 
\begin{enumerate}

\item If $g\,\geq\, 1$, then there is a
closed complex analytic subset $\mathcal{Y}\, \subset\,
\mathrm{Teich}_{\mathfrak h, g,m}$ of codimension at least $g$ such that for any $t\,\in\,
\mathrm{Teich}_{\mathfrak h, g,m} \setminus \mathcal{Y}$, the holomorphic principal $G$--bundle
$\mathcal{E}^t_G$ is semi-stable.

\item If $g\,\geq\, 2$, then there
is a closed complex analytic subset $\mathcal{Y}'\, \subset\,\mathrm{Teich}_{\mathfrak h, g,m}$ of
codimension at least $g-1$ such that for any $t\,\in\, \mathrm{Teich}_{\mathfrak h, g,m}\setminus
\mathcal{Y}'$, the holomorphic principal $G$--bundle $\mathcal{E}^t_G$ is stable.
\end{enumerate}

\end{theorem}
 
\begin{proof}
Let $g\, >\, 1$. Let ${\mathcal Y}\, \subset\, {\mathcal T}_{{\mathfrak h}, g,m}$ denote the (finite) union of all 
Harder-Narasimhan strata ${\mathcal Y}_\kappa$ as in Lemma \ref{lemNitsure} with 
non-trivial Harder-Narasimhan type $\kappa$. From Lemma \ref{lemNitsure} we know
that ${\mathcal Y}$ is a closed complex analytic subset of ${\mathcal T}_{{\mathfrak h}, g,m}$.

Take any $t\, \in\, \mathcal{Y}_\kappa\, \subset\, {\mathcal Y}$. Let $E_G\,=\,
\mathcal{E}_G\vert_{{\mathcal C}_t}$ be the holomorphic
principal $G$--bundle on $$X\,:=\, {\mathcal C}_t\, .$$ The holomorphic
connection on $E_G$ obtained by restricting the universal isomonodromy connection
 will be denoted by $\nabla$. Since $E_G$ is not semistable,
there is a proper parabolic subgroup $P\, \subsetneq\, G$ and a holomorphic
reduction of structure group $E_P\, \subset\, E_G$ to $P$, such that $E_P$ is the
Harder--Narasimhan reduction \cite{Be}, \cite{AAB}; the type of this
Harder--Narasimhan reduction is $\kappa$. From Lemma \ref{lemNitsure} we
know that $E_P$ extends along its stratum to a holomorphic reduction of structure group of the principal
$G$--bundle $\mathcal{E}_G\vert_{\tau^{-1}(\mathcal{Y}_\kappa)}$ to the subgroup $P$.

Let $\mu_{\rm max}$ be the maximal slope (degree/rank) of the terms of the Harder 
Narasimhan-filtration.

We have
\begin{equation}\label{deg2}
\mu_{\rm max}(E_P({\mathfrak g}/{\mathfrak p}))\, <\, 0
\end{equation}
\cite[p.~705]{AAB}. In particular
\begin{equation}\label{deg3}
\text{degree}(E_P({\mathfrak g}/{\mathfrak p})) \, <\, 0\, .
\end{equation}

Form the irreducibility of the connection, we know that the second fundamental form $S(\nabla)$ does not vanish, and so we can build the line bundle $\SL$ as above in \eqref{l}.
$$
{\mathcal L}\, \subset\, E_P({\mathfrak g}/{\mathfrak p}).
$$
{}From \eqref{deg2} we have
\begin{equation}\label{deg}
\text{degree}({\mathcal L}) \, <\, 0\, .
\end{equation}
Therefore, $\mathrm{h}^0(X,\,\SL) =0$, and $\mathrm{h}^1(X,\,\SL)\geq g$.

On the other hand, Lemma \ref{surj1} and Lemma \ref{surj2} gave a surjective map from our deformation space $\H^1_\parallel(X,\, \SC)$ to our obstruction space $\mathrm{H}^1(X,\,\SL)$. The space ${\mathcal Y}$ is thus of codimension at least $g$.

For the second case, when stability fails, the degree of $\SL$ is only less than or equal to zero, and so $\mathrm{h}^1(X,\,\SL)\geq g-1$, giving the announced codimension.
\end{proof}



\begin{thebibliography}{ZZZZ}

\bibitem[AAB]{AAB}B. Anchouche, H. Azad and I. Biswas, Harder-Narasimhan
reduction for principal bundles over a compact K\"ahler manifold, \textit{Math.
Ann.} \textbf{323} (2002), 693--712.

\bibitem[AB]{Bolibruch1} D. Anosov and A. Bolibruch, { The Riemann-Hilbert 
problem}, {\it Aspects of Mathematics}, E22. Friedr. Vieweg \& Sohn,
Braunschweig, 1994.

\bibitem[At]{At} M. F. Atiyah, Complex analytic connections in fibre
bundles, \textit{Trans. Amer. Math. Soc.} \textbf{85} (1957), 181--207.

\bibitem[Be]{Be} K. A. Behrend, Semistability of reductive group schemes over
curves, \textit{Math. Ann.} \textbf{301} (1995), 281--305.

\bibitem[BHH]{BHH} I. Biswas, V. Heu, J. Hurtubise, Isomonodromic deformations of 
logarithmic connections and stability, \textit{Math. Ann.}
DOI 10.1007/s00208-015-1318-5.

\bibitem[Boa1]{Boa1} P. Boalch, $G$-bundles, isomonodromy and quantum
Weyl groups, \textit{Int. Math. Res. Not.} \textbf{22} (2002), 1129--1166

\bibitem[Boa2]{Boa2} P. Boalch, Quasi-Hamiltonian geometry of meromorphic
connections, \textit{Duke Math. Jour.} \textbf{139} (2007), 369--405. 

\bibitem[Boa3]{Boa3} P. Boalch, Geometry and braiding of Stokes data; fission and
wild character varieties, \textit{Ann. Math.} {\bf 179} (2014), 301--365.

\bibitem[Bol1]{Bolibruch2} A. Bolibruch, On sufficient conditions for the 
positive solvability of the Riemann-Hilbert problem, \textit{Math. Notes 
Acad. Sci. USSR} \textbf{51} (1992), 110--117.

\bibitem[Bol2]{Bolibruch3} A. Bolibruch, The Riemann-Hilbert problem, 
\textit{Russian Math. Surveys} \textbf{45} (1990), 1--58.

\bibitem[De]{Dekkers} W. Dekkers, The matrix of a connection having 
regular singularities on a vector bundle of rank 2 on $\mathbb{P}^1 
(\mathbb{C})$, {\it \'Equations diff\'erentielles et syst\`emes de
Pfaff dans le champ complexe} (Sem., Inst. Rech. Math. Avanc\'ee,
Strasbourg, 1975), pp. 33--43, Lecture Notes in Math., 712,
Springer, Berlin, 1979.

\bibitem[EH]{Helene1} H. Esnault and C. Hertling, Semistable bundles and 
reducible representations of the fundamental group, \textit{Int. 
Jour. Math.} \textbf{12} (2001), 847--855.

\bibitem[EV]{Helene2} H. Esnault and E. Viehweg, Semistable bundles on 
curves and irreducible representations of the fundamental group,
{\it Algebraic geometry: Hirzebruch 70} (Warsaw, 1998), 129--138,
Contemp. Math., 241, Amer. Math. Soc., Providence, RI, 1999.

\bibitem[GN]{Nitsure} S. R. Gurjar and N. Nitsure, Schematic Harder-Narasimhan
stratification for families of principal bundles and lambda modules,
{\it Proc. Ind. Acad. Sci. (Math. Sci.)} {\bf 124} (2014), 315--332.

\bibitem[He1]{Viktoria2} V. Heu, Universal isomonodromic deformations of meromorphic rank 2 connections on curves. \textit{Ann. Inst. Fourier (Grenoble)} \textbf{60} (2010), no. 2, 515--549. 

\bibitem[He2]{Viktoria1} V. Heu, Stability of rank $2$ vector bundles 
along isomonodromic deformations, \textit{Math. Ann.} \textbf{60}
(2010), 515--549.


\bibitem[JMU]{JMU} M. Jimbo, T. Miwa and K. Ueno, Monodromy preserving deformation of linear 
ordinary differential equations with rational coefficients. I, \textit{Phys. D} \textbf{2} 
(1981), 306--352.

\bibitem[Ko]{Kostov} V. Kostov, Fuchsian linear systems on $\mathbb{CP}^1$ 
and the Riemann-Hilbert problem, \textit{Com. Ren. Acad. Sci. Paris}
\textbf{315} (1992), 143--148.
 
\bibitem[Ma]{Ma} B. Malgrange, Sur les d\'eformations isomonodromiques I, II,
 Mathematics and physics (Paris, 1979/1982), 427--438, \textit{Progr. Math.}, \textbf{37}, Birkh\"auser Boston, Boston, MA, 1983.

\bibitem[Pl]{Plemelj} J. Plemelj, {Problems in the sense of Riemann and
Klein}, {\it Interscience Tracts in Pure and Applied Mathematics}, {\bf 16},
Interscience Publishers John Wiley \& Sons Inc., New York-London-Sydney, 1964.

\bibitem[Sa]{Sa} C. Sabbah, D\'eformations isomonodromiques et vari\'et\'es de Frobenius 
(French), {\it Savoirs Actuels, Math\'ematiques}, EDP Sciences, Les Ulis; CNRS \'Editions, Paris, 
2002. xvi+289 pp.

\bibitem[Sh]{Sh} S. S. Shatz, The decomposition and specialization of algebraic
families of vector bundles, {\it Compositio Math.} {\bf 35} (1977), 163--187.
\end{thebibliography}
 \end{document}